\documentclass[11pt,a4paper, leqno]{article}
\pdfoutput=1
\usepackage{epsfig}

\usepackage{amsmath,amsthm,amssymb,latexsym}
\usepackage{amsfonts}
\usepackage[american]{babel}
\usepackage{mathrsfs}
\usepackage{color}

\newtheorem{theorem}{Theorem}[section]
\newtheorem{definition}[theorem]{Definition}
\newtheorem{lemma}[theorem]{Lemma}

\newtheorem{rem}[theorem]{Remark}
\newtheorem{proposition}[theorem]{Proposition}

\newtheorem{example}{Example}
\newtheorem{corollary}[theorem]{Corollary}

\newcommand{\diam}{{\mathrm{diam}}\,}

\newcommand{\R}{{\Bbb R}}

\newcommand{\E}{{\Bbb E}}
\newcommand{\N}{{\Bbb N}}

\def\j1n{j=1,\dots,n}
 \def\j1m{j=1,\dots,m}
\def\i1np1{\in +1}

\def\R{\mathbb{R}}
\def\N{\mathbb{N}}

\def\i1np1{\in +1}

\def\R{\mathbb{R}}

\def\u1{u^{(1)}}

\def\h1{h^{(1)}}

\newcommand{\la}{\lambda}

\newcommand{\Om} {\Omega}

\newcommand{\pa}{\partial}

\newcommand{\dist} {\mbox{\rm dist }}

\begin{document}

\begin{title}{On  the complements  of union of open balls of fixed radius in the Euclidean space}
\end{title}

\author{Marco Longinetti\footnote{Marco.Longinetti@unifi.it, Dipartimento DIMAI,
Università degli Studi di Firenze, V.le Morgagni 67/a,  50138
Firenze - Italy}\\
 Paolo Manselli\footnote{Paolo.Manselli@unifi.it, Firenze, Italy}\\
Adriana Venturi\footnote{Adriana.Venturi@unifi.it, Firenze, Italy}}
\date{}
\maketitle

\begin{abstract}Let an $R$-body be  the complement of the union of open balls of  radius $R$ in $\E^d$. The $R$-hulloid of a closed not empty set $A$, the minimal $R$-body containing  $A$, is investigated; if $A$ is the set of the vertices of a simplex,  the $R$-hulloid of $A$ is completely described (if $d=2$)  and  if $d> 2$ special examples are studied.  The class of $R$-bodies   is  compact in the Hausdorff metric if $d =2$, but not compact if $d>2$.    
\end{abstract}

\section{Introduction}Given a closed set $E\subset \E^d (d \geq 2)$, the convex hull of $E$ is the intersection of all closed half spaces containing $E$; the convex hull can be considered as a regularization of $E$. Given $R > 0$, a different hull of $E$ could be  the intersection of all closed sets, containing $E$, complement of open balls of radius $R$ not intersecting $E$. Let us call this set the $R$-hulloid of $E$, denoted as $co_R(E)$; the $R$-bodies are the sets coinciding with their  $R$-hulloids. $R$-bodies are called $2R$-convex sets in \cite{Per}.

  The $R$-hulloid $co_R(E)$ has been introduced by Perkal \cite{Per} as a regularization of $E$, hinting that $co_R(E)$ is a mild regularization of a closed set. Mani-levitska \cite{mani} hinted that the $R$-bodies cannot be too irregular.
 
In our work it is shown that this may not be true: in Theorem  \ref{rhulloidedisconesso} an example of a connected set is  constructed  with disconnected $R$-hulloid.
A deeper study gave us the possibility to add new properties to the $R$-bodies: a representation of $co_R(E)$ is given in Theorem \ref{lemminopaolo} and  new properties of $\pa co_R(E)$ are proved in Theorem \ref{l1}, Theorem \ref{l2} and  Corollary \ref{l3}. Moreover contrasting results on regularity are found:    every closed set contained in an hyperplane or in a sphere of radius $r\geq R$ is an $R$-body (Theorems \ref{r5} and \ref{r4}). As a consequence a problem of Borsuk, quoted by Perkal \cite{Per}, has a negative answer (Remark \ref{r5.1}). In $\S$ \ref{proprbodies} it is shown that the $R$-body regularity heavily depends on the dimension. A definition (Definition \ref{defnotclassic}) similar to the classic convexity is given  for the class of planar $R$-bodies, namely (Theorem \ref{proPcor}):
  $$ A \mbox{ \quad is an $R$-body  iff \quad} co_R(\{a_1,a_2,a_3\}) \subset A \quad  \forall a_1,a_2,a_3 \in A.$$
  As consequence, if $d=2$: a sequence of compact $R$-bodies  converges in the Hausdorff metric to an $R$-body (Corollary \ref{liminR2}). If $d> 2$, in Theorem 3.16 it is proved that   a sequence of compact $R$-bodies  converges  to an $(R-\epsilon)$-body, for every $0< \epsilon < R$; however, the body limit may not to be an $R$-body as an example in $\S$ \ref{seccon} shows. If $E$ is connected,  properties of connectivity of $co_R(E)$   are investigated in  $\S$ \ref{coneeanddisconne}.
 
  In  \cite[Definition 2.1]{Gol} V. Golubyatnikov and  V. Rovenski introduced the class  $\mathcal{K}_2^{1/R}$. In Theorem \ref{inclusionbodies}  it is proved that   
   the class of $R$-bodies is  strongly contained in $\mathcal{K}_2^{1/R}$. If $d=2$ ,  under additional assumptions, it is also proved that the two classes coincide.

\section{Definitions and Preliminaries}

 Let $\E^d, d \ge 2,$ be  
  the linear Euclidean Space  with unit sphere $\mathcal{S}^{d-1} $; $   A \subset \E^d $ will be called a  \textbf{ body } if $ A  $ is non empty and closed.
   The minimal affine space containing $A$ will be
     $Lin(A)$. The convex hull 
   of $A$ will be $co(A)$; for notations and results of convex bodies, let us refer to  \cite{Schn}.

\begin{definition} Let $A$ be a not empty set.

   $ A_\epsilon := \{x \in  \E^d$ $: \mbox{dist} (A,x) < \epsilon;\}; $   
     $ A'_\epsilon := \{x \in  \E^d$ $: \mbox{dist} (A,x) \ge \epsilon \}; $     
     $ A^-:= A \cup \partial A ; $      
       $ A^c:=  \E^d \setminus A  $ ;      
       $Int(A)=A^-\setminus \pa A.$
       
       $ B(x, r) $ will be the open ball of center  $ x \in  \E^d $ and radius $  r >0 $; a sphere of radius $r$ is  $\pa B(x,r)$.

 \end{definition}

  \begin{proposition}\label{propoiniziale} Let us recall the following facts for reference.
 \begin{itemize}

 \item \textbf{1}  $ \quad A_\epsilon $ is open;  $ \quad A_\epsilon =   (A^{-})_\epsilon \subset  ( A_\epsilon)^{-}.$  

\item  \textbf{2} $ \quad A_\epsilon = \{x \in \E^d$ $: \exists \, a \in A, \mbox{ for which } x\in B(a, \epsilon) \} = \{x \in \E^d $  $ : B(x, \epsilon) \cap A \neq \emptyset,   \}=$ $  \cup_{ a \in A}  B(a, \epsilon) = A + B(0,\epsilon). $

\item  \textbf{3} $ \quad  A'_\epsilon = \{x \in \E^d:$  $\forall a \in A $,  $  x \notin B(a, \epsilon)\,\} =$
 
 $ 
\qquad  = \{x \in \E^d:$  $ B(x, \epsilon) \cap A  = \emptyset   \}  
= \cap_{ a \in A }\{ \; B(a, \epsilon)^ c  \}.$

  \item \textbf{4} Let $A_i, i=1,2$ be non empty sets. Then
  $$\quad A^1\subset A^2  \Rightarrow (A^1)_\epsilon \subset  (A^2)_\epsilon .$$

  \item \textbf{5} \; If $ E $ is non empty, then $ E \subset (E'_R)'_R \subset E_R$, see
  \cite[lemma 4.3]{Col}.
 
 \end{itemize}
\end{proposition}
\begin{definition}(\cite{Fed})
	 If $A \subset \E^d,\, a\in A$, then $reach(A,a)$ is the supremum of all
	numbers $\rho$ such that for every $x\in B(a,\rho)$ 
	there exists a unique point $b\in A$ satisfying
	$|b-x|=\dist(x,A)$.
	Also:
	$$reach(A):=\inf \{reach(A,a): a \in A \}.$$
	\end{definition}

Let  $\; b_1, b_2 \in \E^d, | b_1 - b_2 | < 2R $ and let $  \mathfrak{h} (b_1,b_2)  $ be the intersection of all closed balls of radius $ R $ containing $   b_1, b_2  .$

 \begin{proposition}(\cite[Theorem 3.8]{Col}, \cite{Rat})\label{p1} 
The body	$ A $ has reach $ \ge R $ if and only if $ A \cap   \mathfrak{h} (b_1,b_2) $ is connected
 for every $\; b_1, b_2 \in A, 0 <| b_1 - b_2 | < 2R $.

\end{proposition}
 
 \begin{rem}\label{rhullandrhulloid} The R-hull of a set $ E $ was introduced in  \cite[Definition 4.1]{Col} as the minimal set $\hat{E}$ of reach
 $ \ge R $  containing $ E. $ Therefore if $reach(A) \geq R$ , then $A$ coincides with its $R$-hull. The R-hull of a set E may not exist, see \cite[Example 2]{Col}. \end{rem} 
 
\begin{proposition}\cite[Theorem 4.4]{Col}\label{theorem4.4} Let $A \subset \E^d$. If $reach (A') \geq R$ then $A$ admits $R$-hull $\hat{A}$ and   
$$\hat{A}=(A'_R)'_R.$$
\end{proposition}
 \begin{proposition}\cite[Theorem 4.8]{Col}\label{propco2}
 If $A \subset \E^2$ is a connected subset of an open ball of radius $R$, then A admits R-hull.
  \end{proposition}
 In the Appendix a more  detailed  proof is given. Let us also recall the following result:
 \begin{proposition}\cite[Theorem 3.10]{Col}, \cite{Rat2})\label{prop3.10}
  Let $A\subset \E^d$ be a closed set such that $reach(A) \geq R > 0$.  If $D\subset \E^d$ is a closed set such that for every $a,b \in D$, 
  $\mathfrak{h}(a,b)\subset D$ and $A\cap D\neq \emptyset$, then $reach(A\cap D)\geq R$.
 \end{proposition}

 \section{ R-bodies}
  Let  $ R $ be a fixed positive real number. 
  $ B $ will be any open ball of radius $  R $.
    $B(x)$ will be the open ball of center  $ x \in  \E^d$ and radius $  R $.  
 Next definitions  have been introduced in  \cite{Per}.

 \begin{definition}\label{d1}
 Let $A$ be a body,  $ A $ will be called an R-body  if $\;  \forall y \in A^c,$  there exists an open  ball $ B $ in $E^d$ (of radius $R$) satisfying $ y\in B \subset  A^c. $ 
This is equivalent to say  
 $$   A^c = \cup \{B : B \cap A = \emptyset \}; $$
that is 
 $$   A= \cap \{B^c : B \cap A = \emptyset  \}  .$$
  \end{definition}
Let us notice that for any $r\geq  R$, the  body $(B(x,r))^c$ is an $R$-body.
 \begin{definition}\label{d+1}
  Let $ E \subset \E^d$ be a non empty set. The set
  $$  co_R(E) :=  \cap \{B^c : B \cap E = \emptyset  \}  $$
 will be called the \textbf{R-hulloid} of $ E$. Let $co_R(E)=\E^d$ if there are no balls $B \subset E^c$.
 \end{definition}
 
 \begin{rem} In \cite{Per} the sets defined in Definition \ref{d1} are called $2R$ convex sets and the sets defined in Definition \ref{d+1} are called $2R$ convex hulls. On the other hand Meissner \cite{Mei} and Valentine \cite[pp. 99-101]{val} use the names of $R$-convex  sets and $R$-convex hulls for different families of sets. An $s$-convex set is also defined in \cite[p. 42]{Fenc}.  To avoid misunderstandings  we decided to call $R$-bodies and $R$-hulloids   the sets defined in Definition \ref{d1} and in Definition \ref{d+1}.
   \end{rem}
  \begin{rem}\label{R-hull=R-hulloid} Let us notice that $ co_R(E)$ is an R-body (by definition) and $ E \subset   co_R(E) .$
  Moreover $ A$ is an R-body if and only if  $ A =   co_R(A)$.
  The R-hulloid always exists.
   \end{rem}  
  Clearly every  convex body $E$ is an $R$-body (for all positive $R$) and its convex hull $co(E)=E$ coincides with its $R$-hulloid.
  
   \begin{rem}\label{reach>=RimpliesRbodies} It was noticed in \cite[Corollary 4.7]{Col} and proved in \cite[Proposition 1]{Cue} that, when the R-hull exists,  it coincides with the R-hulloid. If $A$ has $reach$ greater or equal than $R$, then  (see remark \ref{rhullandrhulloid}) $A$ has $R$-hull, which coincides with $A$ and with its $R$-hulloid, then $A$ is an $R$-body. 
\end{rem} 
    
    \begin{proposition}\label{provedin8} Let $E$ be a non empty set.
    The following facts have been proved in \cite{Per}.
  
 \begin{itemize}
 \item \textbf{a} $\quad   co_R(E) =( E'_R)'_R; $
 \item \textbf{b} $\quad E^- \subset   co_R(E)  ;$
 \item \textbf{c} \; Let $\, E^1 \subset E^2 ; $ then   $co_R(E^1) \subset  co_R(E^2); $
 \item \textbf{d} $\quad  co_R(E^1) \cup  co_R(E^2) \subset  co_R(E^1 \cup E^2); $
 \item \textbf{e} $\quad   co_R( co_R (E) ) =  co_R (E); $ 
 \item \textbf{f} \; Let $ A^{(\alpha)} , \alpha \in \mathcal{A}$ be R-bodies, then $\cap _{\alpha  \in \mathcal{A}}$ $A^{(\alpha)} $ is an R-body;
 \item  \textbf{g} \; $\diam E = \diam co_R(E)$;
 \item  \textbf{h} \; If $A$ is an $R$-body then $A$ is an $r$-body for $0<r< R$;
\item  \textbf{i} \; $co_R(E)\subset co(E)$ for all $R> 0$.
 \end{itemize}
    \end{proposition}
    \begin{rem}\label{r2}
 Let $ E$ be a body.  
  From  $ \textbf{c} $ of Proposition \ref{provedin8} it follows that if $ A $ is an R-body and  $ A \supset E, $ then $ A \supset co_R(E)$ and  $co_R(E)$ is the minimal  $R$-body  containing $E$. 
 \end{rem}

   \begin{lemma}\label{kincorE}
  A point $k\in co_R(E)$ if and only if  does not exist an  open ball $B(x,l)\ni k$ with $l\geq R$, $B(x,l) \subset E^C$.     
\end{lemma}
\begin{proof} As $(B(x,l))^c$ is an $R$-body,  the set $co_R(E) \cap (B(x,l))^c \supset E$ would be an $R$-body  strictly included 
  in $co_R(E)$, which is the minimal $R$-body containing $E$.
\end{proof}

  \begin{lemma}\label{r3}
 Let $ E$ be a body. Then 
\begin{equation}\label{coRCER}
co_R(E) \subset E_R.
\end{equation}
Moreover 
  $ (E_R)^{-}  $ may not be an R-body.
 \end{lemma}
 \begin{proof} By {\bf  5} of  Proposition \ref{propoiniziale}, $(E'_R)'_R \subset E_R$ and by {\bf a} of Proposition \ref{provedin8},  the inclusion \eqref{coRCER} follows. Let $ x_0 \in  \E^d,$ $ R< \rho < 2R $ and let  $  E=(B(x_0,\rho))^c.   $
Then   $ (E_R)^{-} $  is  $(B(x_0, \rho -R))^c,  $  not an R-body.\end{proof}
 
\begin{theorem}\label{lemminopaolo} Let $E\subset \E^d$ be a body. 
Then
\begin{equation}\label{lemminopaoloeq}
co_R(E)=E_R\cap \Big(\pa(E_R)\Big)_R'.
\end{equation}
\end{theorem}
\begin{proof} 
Formula \eqref{lemminopaoloeq} can also be written as:
\begin{equation}\label{lemminopaoloeq2}
(co_R(E))^c=E'_R\cup \Big(\pa(E_R)\Big)_R.
\end{equation}
Let $\Om= E'_R\cup \Big(\pa(E_R)\Big)_R$.

Inclusion \eqref{coRCER} implies that $E_R' \subset (co_R(E))^c$. Let us notice that: 
$$\Big(\pa(E_R)\Big)_R=\cup \{B(x): x\in \pa(E_R)\}=\cup \{B(x): \dist(x,E)=R\},$$
then
\begin{equation}\label{pacoR(E)}
\Big(\pa(E_R)\Big)_R\subset \cup \{B(x): \dist(x,E)\geq R\}=(co_R(E))^c.
\end{equation}

 Then  
from \eqref{pacoR(E)}: 
$$\Om\subset (co_R(E))^c$$
holds too.

 The open set $(co_R(E))^c$ is the union of the balls $B(x)$, satisfying $B(x) \cap E = \emptyset $; clearly $\dist (x,E) \geq R$; if $\dist(x,E) =R$ then $x\in \pa(E_R)$ and $B(x) \subset \Big(\pa(E_R)\Big)_R$; if $\dist (x,E)>R$ , then $B(x) \subset E'_R$. Therefore
$$(co_R(E))^c \subset \Om.$$

Then $\Om= (co_R(E))^c$.\end{proof}

\begin{rem}The previous theorem is the analogous,  for the $R$-hulloid, of the  property of the convex hull of a body $E$: $co(E)$ is the intersection of all closed half spaces  supporting $E$.
\end{rem}

If $E$ is a compact set, part of the following theorem has been proved in  \cite[Proposition 2]{Cue}.

  \begin{theorem}\label{l1}
 Let $ E $ be a body, $ k \in co_R(E)$,   $l=inf_ {x \in  E'_R } | k-x |=\dist(k, E'_R).$ Then $ l  $ is a minimum and $ l \ge R. $ Moreover $ l= R    $ if and only if  $ k \in \partial co_R(E)  $ and there exists $ x_0 \in E'_R  $ satisfying $ B(x_0) \subset E^c, \; \partial B(x_0) \ni k . $
\end{theorem}
 
 \begin{proof}
 As $ co_R(E) = \cap\{B^c: B^c \supset E\},   $ then $\dist(E'_R,co_R(E)) \ge R.$  Let $ x_n \in  E'_R $ satisfying $ | x_n - k| \rightarrow l \ge R; $ by possibly passing to a subsequence, one can assume that $ x_n \rightarrow x_0  \in   E'_R, $ where $ |x_0-k| = l. $ 
 If  $|x_0 - k |= R $  then $ k \in  co_R(E) \cap \partial B(x_0)   $. As $ l=R $, it  cannot be  $ k \in int( co_R(E))$.  Therefore the thesis holds.   \end{proof}

 \begin{theorem}\label{l2}
  Let $ E $ be a body, $\; k \in \partial co_R(E).$ Then there exists $ B \subset E^c $ satisfying $ k \in \partial B$. Moreover if $\mathfrak{F}=\{B \subset E^c: \pa B \cap co_R(E)\neq \emptyset\}$, then $\mathfrak{F}$ is not empty and if $B \in \mathfrak{F}$ then $\pa B\cap E \neq\emptyset$.
   \end{theorem}
 \begin{proof}  
 If $  k \in \partial co_R(E), $ by previous theorem there exists $ x_0 \in E'_R  $ with the property $ B(x_0) \subset E^c, \; \partial B(x_0) \ni k . $ If $\dist (x_0, E)= l> R  $ then $k \in B^1= B(x_0,l)\subset E^c$, this is impossible by Lemma
\ref{kincorE} and $\mathfrak{F}$ is non empty. 
Let  $B(x) \in \mathfrak{F}$ and, by contradiction, let $\pa B \cap E= \emptyset$; then, $R_1= \dist (x,E) > R$. Thus $B(x,R_1)^c $ is an $R$-body containing $E$, then  $co_R(E) \subset B(x,R_1)^c$; as $\pa B(x) \subset B(x,R_1)$ so $\pa B(x)\cap co_R(E)= \emptyset$, contradiction with $B(x) \in \mathfrak{F}$.
  \end{proof}

 \begin{corollary}\label{l3}
 Let $ A$ be an R-body. Then : \par
  (i) 
  $ \Xi(A) := \{ x: B(x) \subset A^c \}  $ (the set of centers of balls of radius R contained in $A^c$)
 is closed; \par (ii) $\forall y \in \partial A, $ there exists $ x_0 \in \Xi(A) $ with the property: $ y\in \partial B(x_0). $
 \end{corollary}
\begin{proof}  
 Let $ x_0 $  be an accumulation point of $ \Xi(A)$ and $  \Xi(A) \ni x_n \rightarrow x_0;$ let $ b\in B(x_0),$ then $  \lim_{n \rightarrow \infty}|b- x_n| = |b- x_0| $ where  $ |b-x_0| <R.  $ Thus for  $n $ sufficiently large $ | b- x_n|<R,  $ therefore  $ b\in B(x_n) \subset A^c, \forall b\in B(x_0) $.  Then $ B(x_0) \subset A^c, $ $ x_0 \in \Xi (A)  $ and (i) holds.
 
(ii) follows by Theorem \ref{l2}.
\end{proof}

\begin{lemma}\label{r5.11}
	Let $ A $ be a body; if $A^c$  is union of closed balls of radius $ R $, then $ A $ is an $R$-body. 
\end{lemma}
\begin{proof}
	For every $ y \in A^c $ there exists $(B(z))^- \subset A^c$,   $ y \in (B(z))^- $. As $A$ and  $ (B(z))^-  $ are closed and disjoint, there exists $  R_1 > R $ so that $B(z,R_1) \subset A^c $. Then there exists a ball $ B \subset A^c, B\ni y $. Thus $ A^c $ is union of open balls of radius $ R $ 
	and A is an $R$-body. \end{proof} 
Let us notice that there exist  $R$-bodies $A$ such that $A^c$ is  not union of closed balls of radius $R$. As example, let $A=B^c$.

\begin{theorem}\label{r5.12}
	Let $ A $ be a body, not an R-body. Then there exists $ y_0 \in A^c $ such that 
	$y_0$ belongs to no closed ball  of radius $ R $, contained in $  A^c $.
\end{theorem}
\begin{proof} 
	By contradiction, let us assume that every $ y \in A^c $ is contained in a closed ball of radius $ R $ contained in $ A^c$,  then $ A^c$ is union of closed balls of radius $ R $ and satisfies the hypothesis of Lemma \ref{r5.11}, then $ A $ is an R-body. Impossible.
\end{proof}

Let  $\mathcal{C}^d$ be the metric space of  the compact bodies in $\E^d $ with the Hausdorff distance $ \delta_H(F,G):= $ min $\{\epsilon \ge 0: F \subset G_{\epsilon},$ $  G \subset F_{\epsilon}\}.   $

From a bounded sequence in  $\mathcal{C}^d$ one can select a convergent subsequence  in the Hausdorff  metric (see e.g.  \cite[Theorem 1.8.4]{Schn}).

Let  $\mathcal{R}^d = \{ A\subset \mathcal{C}^d : A $ is  an  R-body $  \} $. 
  Let $  A\subset  \E^d $ be a body, $ \epsilon >0 $. Let $ A_{\epsilon}^-:= \{x\in E^d: dist(A,x)\le\epsilon\}=(A_\epsilon)^-.  $  $D =B^-$ will be any closed ball of radius $R$.
\begin{theorem}\label{compatRn}  Let  $ A^{(n)} $ be a sequence of compact R-bodies; let us assume   that $A^{(n)} \rightarrow A \in \mathcal{C}^d$ in the Hausdorff metric. Then, $A$ is an $R_\epsilon$-body 
  for every $0 <R_\epsilon  < R $.	
 \end{theorem}
 \begin{proof}		
	 By contradiction, let us assume  that $ A $  it is not an $R_\epsilon$-body. Then by Theorem \ref{r5.12}, there exists $ y_0 \in A^c $ with the property 
	 \begin{equation}\label{propy0}
	 \mbox {$y_0 $ belongs to no  closed ball, of radius $R_\epsilon $, subset of $ A^c$.} 
\end{equation}

	 As $\dist(y_0,A) > 0$ , then     $ y_0\in (A_{\sigma})^c $ for suitable $ \sigma >0.$ 	 
	 As $ A^{(n)} \rightarrow A $ in the Hausdorff metric,  there exists a sequence $ \epsilon_n \rightarrow 0^+ $ satisfying $ A^{(n)} \subset A_{\epsilon_n}$, $ A \subset A_{\epsilon_n} . $ 	   
	     For n sufficiently large $(A_{\sigma})^c \subset (A_{\epsilon_n})^c \;   $
	      and  $ y_0 \in (A_{\epsilon_n})^c \subset ( A^{(n)})^c. $  As $ A^{(n)} \in \mathcal{R}^d, $ then there exist open balls $  B({x_n}) $ satisfying $  y_0 \in  B({x_n}) \subset ( A^{(n)}) ^c;  $ then
	       $\dist (x_n,A^{(n)})\ge R, \;$ $\dist (x_n,A)\ge R- \epsilon_n$ .
	 
	  As $ |x_n - y_0| < R,   $ 
	 by possibly passing to a subsequence,   $ x_n \rightarrow x_0 \in  \E^d.$
	 The point $ x_0 $ satisfies: $ |x_0 - y_0| \leq R ,$  $\dist (x_0,A)\ge R.$	 
 Then $B({x_0})\subset A^c$ and $ D:=(B({x_0}))^- $ is a closed ball of radius $ R $ containing $y_0 $. If $y_0\in B({x_0})$, then $D_\rho =B^-(x_0, \rho) $, with  $\rho = max\{|y_0-x_0|, R_\epsilon \}$ is a closed ball which provides a contradiction with \eqref{propy0}. In case $y_0\in \pa B({x_0})$ the closed ball enclosed in $D$,  tangent to $ \pa B({x_0})$ at $y_0$, with radius $R_\epsilon $, provides a contradiction with property  \eqref{propy0}. \end{proof}

 \begin{rem}In section \ref{seccon} it will be  shown that in $\E^3$  a limit (in Hausdorff metric) of a sequence of $R$-bodies may  be not an $R$-body. In Corollary \ref{liminR2} it will be  proved that,  in $\E^2$, a limit of a sequence of $R$-bodies (in Hausdorff metric) is an $R$-body too.
 
 \end{rem}
  
  \begin{theorem}\label{r5}
  
  Let  $ \Sigma=\pa B(r) \subset \E^d $  be a sphere of radius $ r \ge R$ and let   $ E $ be a body subset of $ \Sigma. $  Then $ E $ is an R-body.
  \end{theorem}

   \begin{proof}
  $\Sigma $ is a topological space with the topology induced by $  \E^d $ and  $ E  $  is closed in that topology.  Then $ \Sigma \setminus E $ is union of (d-1)-dimensional open balls in $ \Sigma $.  Let $D=(B(r))^-$, as 
  $  \E^d \setminus \Sigma = B(r) \cup D^c $, then
 $$  \E^d \setminus E = B(r) \cup D^c \cup ( \Sigma \setminus E) $$
  is union of the following open balls of radius $ R $: \par
 (i) all open  balls of radius  $ R $ contained in $B(r) $, which fill $B(r)$ since $r\geq R$;
  \par
   (ii)
 all open balls $ B $ of radius  $ R $ contained in  $  D^c   $; \par 
 (iii) all open balls $ B $ of radius  $ R $
  satisfying the 
 property: $ B \cap \Sigma $ is a (d-1)-dimensional open  ball in  $ \Sigma \setminus E $.
 
  So $ E $ is an R-body.
 \end{proof}
  
  With a similar proof, the following fact can be proved.
 
 \begin{theorem}\label{r4}
 Let $  E \subset  \E^d $ be a body, subset of a hyperplane $ \Pi. $ Then $ E $ is an R-body.
 \end{theorem}

   \begin{rem}\label{r5.1}
 In \cite{Per}, p.9,  a question of Borsuk was stated:
 'Are the $R$-bodies  locally contractible?'.
  
  The Borsuk's question has a negative answer: let $ \pi $ be an hyperplane in $E^d$. By Theorem \ref{r4} every   body, subset of $\pi $, is an $R$-body; then there exist not locally contractible bodies subsets of $\pi$. 
  
 \end{rem}

\section{Properties of  R-bodies in $\E^2$.}\label{proprbodies}
 \subsection{$R$-hulloid of three points in $\E^2$.}
  Let  $ R $ be a fixed positive real number. 
Let $T$ be a not degenerate triangle in $\E^2$, $V= \{x_1,x_2,x_3\}$ be the set of its vertices, $r(V)$ be the radius of the circle circumscribed to $T$. By Theorem \ref{r5},  if $r(V) \geq R$,  then $co_R(V)=V$.

\begin{proposition}\label{propjohnson}
Let $\{x_1,x_2,x_3\}$ be the vertices of a triangle $T$ inscribed in a circumference $C$  of radius $r$. Three  possible cases hold:
\par  
\item[i)] (\cite[pag 16]{johnson2}) if $T$ is acute-angled then the three circumferences of radius $r$, each one  through two vertices of $T$, different from $C$,  meet in the orthocenter $y$ of $T$;
\item[ii)] if $T$ is obtuse-angled in $x_3$  then the two circumferences of radius $r$ through the vertices $\{x_1, x_3\}$ and $\{x_2, x_3\}$, respectively, different from $C$, meet $C$ in $x_3$ and in a point exterior to T;
\item[iii)] if $T$ is  right-angled then   the two circumferences of radius $r$ through the vertices $\{x_1, x_3\}$ and $\{x_2, x_3\}$,  different from $C$, are tangent at $x_3$.
\end{proposition}
\begin{proof} i) it is  related to the  Johnson's Theorem \cite{johnson};  ii) and iii) follows by construction.
\end{proof}

\begin{theorem}\label{cor(v)inplane}
Let  $V=\{x_1,x_2,x_3\}$ be the set of the vertices of a triangle $T$ with circumradius $r=r(T)$.   If $r(T) < R$, then
$$co_R(V)=V\cup \tilde{T},$$
where  $\tilde{T}\subset T $ is the curvilinear triangle  bordered by three arcs of circumferences of radius $R$, each one through two vertices of $T$. If $T$ is a right-angled or obtuse-angled then the vertex of the major angle of $T$ is also a vertex of $\tilde{T}$.
\end{theorem}
\begin{proof}  Let 
  $B(q_i,r )$,$B(c_i,R )$  be the open circles, not containing $x_i$,  with boundary  through the two vertices of $T$ different from $x_i, i= 1,2,3$. In the case i) of Proposition \ref{propjohnson}, the orthocenter $y$ of $T$ is in the interior of $T$ and $y \in \cap_{i=1,2,3} \pa B(q_i,r)$. As  $R>r$ : 
$ T\cap B(c_i,R) \subsetneq T\cap B(q_i,r)$, then
 $\dist (y, B(c_i,R)) > 0$, ($i= 1,2,3$). Thus
 \begin{equation}\label{Ttilde=}
 \tilde{T}:\equiv T \cap \left(\bigcup_{j =1} ^{3} B(c_j, R) \right)^C 
 \end{equation}
is a curvilinear triangle with $y\in Int( \tilde{T})$; moreover $\pa\tilde{T}$ is union of  of three arcs of the circumferences  $\pa B(c_i,R)$ $(i=1,2,3)$.

If $T$ is obtuse-angled at $x_3$, 
case ii) of Proposition \ref{propjohnson}),   the two circumferences $\pa B(c_i,R )$ containing $x_3$ and another vertex of $T$ cross each other in $x_3$ and in a point exterior to $T$.
If $T$ is right-angled at $x_3$ the two  circumferences $\pa B(q_i,r )$   meet and are tangent to each other in $x_3$, then  
again  the circumferences $\pa B(c_i,R )$  cross each other in $x_3$ and in a point exterior to $T$. In both cases $\dist (x_3, B(c_3,R)) > 0$ and $\tilde{T}$, given by \eqref{Ttilde=}, is a curvilinear triangle with a vertex at $x_3$. \end{proof}

\subsection{ Two dimensional $R$-bodies, equivalent definitions}\label{equivdefini}

\begin{definition}\label{defnotclassic}
Let $a_1, a_2$ be two  points in $\E^2$, with $0 < |a_1-a_2| < 2R$. Let $B(x_1), B(x_2) $ the two open circles  with the boundaries through $a_1,a_2$. Let us define
$$H(a_1,a_2)= B(x_1) \cup B(x_2).$$
\end{definition}

\begin{definition}
 Let $A$ be a plane body. $A$  satisfies the  property $\mathfrak{Q}_R$ if :
$$\forall a_1,a_2,a_3 \in A \mbox{\quad the $R$-hulloid of  the set  $\{a_1,a_2,a_3\}$  is a subset of } A. $$ 
\end{definition}
\begin{lemma}\label{Qrproperty}Let $A$ be a plane body. If $A$ satisfies the property  $\mathfrak{Q}_R$,  then 
\begin{equation}\label{Prproperty}
\{a_1,a_2\}\subset A, 0 < |a_1-a_2| < 2R : h(a_1,a_2)\setminus \{a_1,a_2\}\subset A^c  \Rightarrow H(a_1,a_2) \subset A^c.
\end{equation}

\end{lemma}
\begin{proof} Let  $H(a_1,a_2)= B(x_1) \cup B(x_2)$. Let assume, by contradiction, that there exist $a_3 \in A \cap (B(x_1) \setminus h(a_1,a_2))$. Let $T= co(\{a_1,a_2,a_3\})$, then  $r(T) < R$. By Theorem \ref{cor(v)inplane} 
there exist $y_1, y_2 \in arc_{\pa B(x_2)}(a_1,a_2) $ satisfying
$$arc_{\pa B(x_2)}(y_1,y_2)\subset co_R(\{a_1,a_2,a_3\})\subset  A.$$
As 
$$arc_{\pa B(x_2)}(y_1,y_2)\subset h(a_1,a_2)\setminus \{a_1,a_2\} \subset A^c, $$
this is impossible. The proof is similar if $a_3 \in B(x_2)$.
\end{proof}

\begin{theorem}\label{proPcor} Let $A$ be a plane  body. $A$ is an $R$-body if and only if $A$ satisfies the   property $\mathfrak{Q}_R$.
\end{theorem}

\begin{proof}
Let  $A$ be an $R$-body then $co_R(\{a_1,a_2,a_3\}) \subset co_R(A)=A$ and $\mathfrak{Q}_R$ holds for $A$. 

On the other hand  let assume the property $\mathfrak{Q}_R$ holds for a  body $A$. Let us prove that $A$ is an $R$-body, by showing:
\begin{equation}\label{propy0defRbody}
\mbox{if \quad} y_0\in A^c  \mbox{\quad then  \quad} \exists B \ni y_0,\,  B \subset A^c.
\end{equation}

Let $y_0 \in A^c$, then there exists $\delta > 0$ such that $\dist(y_0,A) =\delta $. If $\delta \geq R$, then $B(y_0, R) \subset B(y_0,\delta) $ and \eqref{propy0defRbody}  holds. Let $\delta < R$.
By definition of $\delta$, there exists $ a_1\in A \cap \pa B(y_0, \delta)$ and $B(y_0, \delta) \subset A^c$. 
 There are two cases:
 \par  
 i) there exists a point $a_2\neq a_1$, $a_2 \in A \cap \pa B(y_0, \delta)$;
 \par
ii)  $A\cap \pa B(y_0, \delta)=\{a_1\}$. 
 
In the case  i),  $h(a_1,a_2)\setminus \{a_1,a_2\} \subset B(y_0, \delta) \subset A^c$. Let 
 $H(a_1,a_2)=B(x_1)\cup B(x_2)$; by Lemma \ref{Qrproperty} the following inclusion holds:
\begin{equation}\label{claimH}
 H(a_1,a_2) \subset A^c.
\end{equation} 
 As $y_0 \in B(x_1)$ or $y_0\in B(x_2)$ and both balls $B(x_1), i=1,2$ have empty intersection with $A$, then  $y_0$ satisfies \eqref{propy0defRbody}.

In the case ii) on $\pa B(y_0,\delta)$ let  $a_*$ be the symmetric point  of $a_1$ with respect to the center $y_0$. For $t> 2 $ let $a(t)=a_1+(t-1)(a_*-a_1)$.
Let $t_R> 2 $ such that $|a_1-a(t_R)| = 2R$.  The set function $t \to h(a_1,a(t))\setminus  \{a_1\}$,  for  $2\leq  t <t_R,  $  is  strictly increasing  with respect to the inclusion. If for all $2 \leq  t < t_R$ the set  $h(a_1,a(t)) \setminus \{a_1\}\subset A^c$ then $\lim_{t\to {t_R}-}h(a_1,a(t))$ is a closed ball $D \ni y_0$ of radius $R$,   $A^c \supset Int(D) \ni y_0$ and \eqref{propy0defRbody} holds. Otherwise, there exists $2< \tau < t_R$ satisfying  
    $h(a_1,a(\tau)) \setminus \{a_1\}\cap A \neq \emptyset $.  Let 
    $$\overline{t}= Inf \{t \in [2,t_R] : \Big( h(a_1,a(t)) \setminus \{a_1\}\Big)\cap A \neq \emptyset \} $$
    and let 
     \begin{equation}\label{familyF(t)}
   2 \leq t \leq t_R \rightarrow F(t):= \Big( h(a_1,a(t))\setminus \{a_1\}\Big)\cap (B(y_0, \delta))^c.
   \end{equation}
  
 By construction $\{F(t)\}$   is a continuous family of bodies, strictly monotone with respect to the inclusion,     with 
     $\dist (F(t), A) > 0$ for $t < \overline{t}$. Then $ F(\overline{t})\cap A\neq \emptyset$,  $Int(F(\overline{t})) \subset A^c$ and $\dist (a_1, F(\overline{t}))>0$.
      Therefore there exists 
    $a_2\in \pa  F(\overline{t})\cap \pa A $ of minimum distance from $a_1$. This implies that $arc_{\pa F(\overline{t}) }(a_1, a_2)$ has no interior points of the body $A$. 
 Then,  
$h(a_1,a_2)\setminus \{a_1,a_2 \}\subset  A^c$; by  arguing as in case i), the inclusion
 $\eqref{claimH}$ holds
      and $y_0$ satisfies \eqref{propy0defRbody}.   
 \end{proof}

\begin{theorem}\label{teoremaQrhoQR}  Let $A \subset \E^2$ be a body. If $A$ is a $\rho$-body for $\rho < R$ then $A$ is an $R$-body.
\end{theorem} 
\begin{proof}
 If $A$ is $\rho$-body  the property $\mathfrak{Q}_{\rho}$ holds for $\rho < R$.  Let us show that it holds for $\rho=R$. Let $ a_1,a_2,a_3 \in A $, with $r(\{a_1,a_2,a_3\})\geq R$, then 
 $co_R(\{a_1,a_2,a_3\})= \{a_1,a_2,a_3\}\subset A$. In case  $r(\{a_1,a_2,a_3\})< R$ let  $\rho > r(\{a_1,a_2,a_3\})$; by Theorem \ref{cor(v)inplane}, with $\rho$ instead of  $R$,
 $$co_\rho(\{a_1,a_2,a_3\})= \{a_1,a_2,a_3\}\bigcup \tilde{T_\rho}$$
 with $\tilde{T_\rho}$ a curvilinear triangle subset of $A$, bounded by arcs of radius $\rho$. 
 As $A$ is closed  and $\tilde{T_\rho} \to  \tilde{T}$, then $\tilde{T}\subset A$. Therefore
  $\mathfrak{Q}_{R}$ holds too and previous theorem proves that $A$ is an $R$-body. 
\end{proof}

From Theorem \ref{teoremaQrhoQR}  and Theorem \ref{compatRn} it follows 
\begin{corollary}\label{liminR2} A limit of a sequence of plane  $R$-bodies (in Hausdorff metric) is an $R$-body too.
\end{corollary}

\begin{rem}With  arguments similar to  the proof of Theorem \ref{proPcor}, it can also  be proved that  
for a plane body $A$ the property $\mathfrak{Q}_{R}$ is equivalent to the  property \eqref{Prproperty}.
\end{rem}

\subsection{Connected and disconnected $R$-bodies in $\E^2$}\label{coneeanddisconne}

 \begin{theorem}\label{conn=2}
 	Let $ E $ be a connected body in $  \E^2 $,  contained in an open ball $B$ of radius $R$; then $ co_R(E)  $ is  connected.
 \end{theorem}
 \begin{proof} 
 As $E$ is connected, by Proposition \ref{propco2}, $E$ admits $R$-hull $A$ of $ reach \geq R$;  then,   by Remark
 \ref{R-hull=R-hulloid},
 $A =co_R(E)$. 
 By Proposition \ref{p1} the set $A$ is connected.
  \end{proof}
 In the previous theorem the assumption that $E$ is contained  in an open ball of radius $R$ is needed as the following example shows.
\begin{example}In $E^2$ let $ \Sigma_0  := \partial B(0, R_0),  $ with $$\frac{R}{\sqrt{3}}< R_o < R.$$
Let $k_i \in \Sigma_0, i=1,2,3$ be the vertices of an equilateral triangle  $T$ and let $\pa B(o_j, R) $ the circumference,   through  the two points $k_i, i\neq j$, with $k_j\not \in B(o_j,R)$. Let 
$ D:= (B(0,4R))^-$ and 
$$E := D \cap \left(B(0,R_0)  \bigcup_{j =1} ^{3} B(o_j, R) \right)^c.$$
Then $E$ is a plane  connected body  with disconnected $R$-hulloid.
\end{example}
 \begin{proof}
It is obvious that $E$ is connected since  and it is  homotopic to a ring. $E^c$ is an open set since 
$E^c$ is the union of  $D^c$ and open balls. 
As   $R_0< R$ and $\forall i\neq j, k_i \in \pa B(o_j), k _j \not \in B(o_j)$ the set $E^c$  does not contain the set of the  vertices $k_i$.  
 Let 
 $$\tilde{T}:\equiv  \left(\bigcup_{j =1} ^{3} B(o_j, R) \cup B(0,R_0)\right)\setminus 
 \left(\bigcup_{j =1} ^{3} B(o_j, R) \right).$$
 $\tilde{T}$ is a curvilinear triangle and it is a closed  connected set disjoint from $E$; moreover  any point of $\tilde{T}$ can not lie in an open circle of radius $R$ avoiding all  the vertices $k_i$ of the equilateral triangle $T$.
 Then, by Lemma \ref{kincorE}, $E \cup\tilde{T}\subset co_R(E) $; as the complementary of $E \cup\tilde{T}$ is $D^c\cup_j B(o_j,R)$, union of open balls of radius $R$, then   $E \cup \tilde{T}$ is an $R$-body, that is 
 $$co_R(E) = E \cup \tilde{T}$$
 which is a disconnected  $R$-body.
 \end{proof}
The previous example can be modified to get a simply connected set  $E_*$ such that $co_R(E_*)$ is disconnected. Let us consider $E_*=E\cap W^c$, where $W$ is a small strip from $\pa B(o_1,R)$ to $ \pa D( 4,R)$. Clearly $co_R(E_*)=co_R(E)$ is disconnected and $E_*$ is a simply connected set.

 \section{$R$-hulloid of the vertices of a simplex in $\R^d$ }\label{seccon}
 
 \begin{definition}\label{definiz5}  Let $d\geq 2$, $1 \leq n \leq d $. Let $\{v_1, \ldots, v_{n+1}\} \subset \R^d$ be a family of affinely indipendent points and let  $V=\{v_1, \ldots, v_{n+1}\}\subset \R^d$.
 An $n$-simplex is the set $T=co(V)$.

 Let $T=co(V)$;   the $(d-1)$-simplexes  $T_i=co(V\setminus \{v_i\}), (i=1,\ldots , d+1$) are called the facets of $T$. If $V$ lies  on a sphere, centered in $Lin(T)$, and its points are equidistant, then  $T$ will be called a regular simplex.
 \end{definition}

 It is well known the following fact:  let $V$ the set of the vertices of a $d$-simplex $T$ in $\E^d$. There exists a unique open ball $B(V)$ such that  all the vertices in $V$  belong to $\pa B(V)$, called the circumball to $co(V)$. Let us notice   that $D(V)=(B(V))^-$  does not coincide (in general) with the closed ball of minimum radius containing $V$, as an obtuse isosceles triangle shows.

 \begin{definition}Let $1 < n \leq d$; if $T$ is a $n$-simplex, the circumcenter $c(T)$ and the circumradius $r(T)$ are the center and the radius respectively, of the unique open ball  $ B(c(T), r(T))$, called circumball of $T$, such  that: i) $c(T)\in Lin(T)$; ii) $ \pa B(c(T), r(T)) \supset V$. 
  \end{definition}
Let us denote $r(V):\equiv r(co(V))$, $c(V):\equiv c(co(V))$; $B(V)=B(c(V),r(V))$.

 From Theorem \ref{r5} it follows that 
\begin{corollary}\label{corollario5.3}If  $r(V) \geq R$ then 
\begin{equation}\label{co_R=E}
co_R(V)=V.
\end{equation}
 \end{corollary}
\begin{definition}\label{defhull} Let $R > 0$. The $R$-hulloid of $V$ will be called  full if its   interior  is not empty.
\end{definition}
 If $d=2$, let $V$ be the set of the vertices of a triangle   with circumradius less than $R$;  by Theorem \ref{cor(v)inplane}, $co_R(V) $ is   full.

\subsection{Examples of $R$-hulloid of the vertices of a simplex in $\E^d$}
Convex sets on $\mathcal{S}^{d-1}$ have been studied in \cite{Santalo}. Here properties of regular simplexes on  $\mathcal{S}^{d-1}$ are recalled and used. If $S$ is a regular simplex, centroid and circumcenter coincide.
\begin{lemma}\label{lemmaSantalo} Let $d >1, R_0 > 0, \Sigma_0  := \partial B(0, R_0) $ in $\E^d$. Let $ W=\{k_1, \ldots , k_{d+1}\}\subset \Sigma_0   $ be the set of the vertices of a regular $d$-simplex $S$ on $ \Sigma_0$. Then 
\begin{equation}\label{scalarkikj-}
<k_i, k_j> =-R_0^2 /d, \quad i\neq j
\end{equation}
and
\begin{equation}\label{distkikj-}
	|k_i-k_j|=\sqrt{2\frac{d+1}{d}}R_0.	
		\end{equation}
	Let	$W_i=W\setminus \{k_i\}$ and  let $\Sigma_i \subset \Sigma_0$  be the $(d-1)$-dimensional sphere  through the points of $W_i$. Then 
		$ \Sigma_i$ has center $ -k_i/d$
and
\begin{equation}\label{distptoW}
\forall p\in \Sigma_0 \mbox{\, the spherical distance on $\Sigma_0$ from \, }  p \mbox{  to \,} W \mbox{\, is less or equal to\, } 
R_0\arccos 1/d.
\end{equation}
 
\end{lemma}
 \begin{proof}
 As the centroid of $S$ is $0$, then
 $$ \sum_{i=1} ^ {d+1} k_i =0, \quad |k_i|^2= R_0^2, \quad <k_i, k_j> =R_0^2 \cos \phi \quad  (i,j =1, \dots ,d+1), i\neq j  $$
		and
		$$  0 = <k_j,  \sum_{i=1} ^ {d+1} k_i>  = (R_0)^2 + d\,( R_0)^2 \cos 	\phi \quad (j =1, \dots ,d+1).$$ 
		Therefore $ \cos \phi =- \frac{1}{d}; $ so \eqref{scalarkikj-} and \eqref{distkikj-} hold.
		
		As $S_i= co(W_i)$ is an equilateral  $(d-1)$-simplex, the centroid of $S_i$  will be $ \frac{1}{d}\sum_{j\neq i}k_j= - k_i/d$ and coincides with the center of $\Sigma_i$.
	Let $\tilde{F_j}$ the spherical $(d-1)$-dimensional ball on $\Sigma_0$ of center $-k_j$ bounded by $\Sigma_j$.	Then $\tilde{F_j}$  has spherical radius 
		$$R_0 \arccos \frac{\langle -k_i, k_j\rangle }{R_0^2} =R_0\arccos 1/d.$$ 
		As  $\cup_{j=1}^{d+1}\tilde{F_j} =\Sigma_0$ 
 the thesis follows. 
\end{proof}

\begin{theorem}\label{lemmasigmao} Let $d > 2$ and  let $S$ be the regular simplex introduced in  Lemma \ref{lemmaSantalo}; let $ R = \frac{d}{2}R_0$.  Then the set $W$ of its vertices  is not an $R$-body and $co_R(W)= W \cup \{0\}$ is not full.
\end{theorem}
\begin{proof}
			Let $ B(o_j, \rho_j) $ with the property that  
$$ \partial B(o_j, \rho_j) \supset 
	 \{ 0,  k_1, \ldots , k_{j-1},  k_{j+1}, \ldots , k_{d+1}  \}  .$$   Clearly
	  $o_j= - \lambda  k_j, (\lambda >0).$	
		As $ |o_j -0|^2 = | o_j- k_i|^2, i\neq j  $ then
$$ (\lambda R_0)^2 = (\lambda R_0)^2 + ( R_0)^2 + 2 \lambda ( R_0)^2 \cos \phi, $$ therefore $ \lambda = \frac{ d}{2},	
	 \; o_j =-  \frac{ d}{2} k_j $ and  $  \rho_j  =|o_j -0| = \frac{d R_0}{2} = R$. 
	 
	  From \eqref{distkikj-} it follows
\begin{equation}\label{distcicj}
	|o_i-o_j|= 2R\sqrt{\frac12+\frac{1}{2d}}, \quad j\neq i.	
		\end{equation}

	Claim $\mathcal{Q}$: {\em Let  $R-R_0< |z|\leq  R$, ${Q}_z:\equiv B(0,R_0)\cap B(z,R)$. Then $\pa {Q}_z \cap \Sigma_0$ is a spherical $(d-1)$ dimensional ball on $\Sigma_0$ of radius $r$.  If $|z| < R$ then }  
		$$r > R_0\arccos1/d.$$ 
		
	Proof:  let $v=z/|z|$, the family of  ${Q}_{\la v}$ is  ordered by inclusion 	  for $R-R_0 < \lambda  \leq R$, with minimum set for  $\lambda=R$; for $\lambda=R $ the spherical $(d-1)$ dimensional ball $\pa {Q}_{R_0z/|z|}$ has radius $R_0\arccos1/d$.
	  
If  $R-R_0 < |z| < R$, then from   Claim $\mathcal{Q}$ and \eqref{distptoW},
 any open ball	$B(z,R)$, which contains  the point  $0$ contains at least one of the vertex $k_i$,  $i=1, \ldots, d+1$.	  
	 As $ 0\not \in W$ the set $W$ is not an $R$-body. Moreover since
$$(W\cup \{0\})^c= 	\bigcup_{j =1}^{d+1} B(o_j, R)\bigcup (co(W))^c,$$
 then $W\cup \{0\}$ is an $R$-body   containing $W$; then $W\cup \{0\}$  is the $R$-hulloid of $W$ and it has empty interior.		  
 \end{proof}

   \begin{theorem}\label{nocompattezzained}In $E^3$ there exist sequences of $R$-bodies with limit, in 
   the Hausdorff metric, a  body that is not an $R$-body.
   \end{theorem}
\begin{proof} Let us use the notations of Lemma \ref{lemmaSantalo} in the special case $d=3$.

 Let $k_i, i=1,\ldots, 4$ the vertices of a regular simplex in $\E^3£$ on  the sphere $ \Sigma_0  := \partial B(0, R_0)$, $R_0=
\frac{2 R}{3}$.

For any fixed  $ i=1, \ldots 4$
the vertices $k_j, j\neq i$ belong to the boundary of 
the ball $B(o_i, R)$, with  
$o_i=-\frac32 k_i$.

From \eqref{scalarkikj-}   it follows that
$$< o_j, k_i> = \frac29R^2,  \quad   i\neq j,   \quad i,j =1, \dots ,4.$$

Let $\epsilon\to 0^+$ and let $x_i^{(n)}= k_i+\epsilon_n\frac{k_i}{|k_i|}, i=1, \ldots,4$. The points $x_i^{(n)}$ are  the vertices of a regular simplex  $T^{(n)}$ in $\E^3$.
For $i\neq j$ let $R_n= |o_i-x_j^{(n)}|$,
then
$$ R_n^2= R^2+\epsilon_n^2+2<k_i-o_j, k_i/|k_i| > \epsilon_n= R^2+\epsilon_n^2+\frac23R\epsilon_n> R^2.$$ For all $n \in \N$ let 
$$W^{(n)}:=\{x_1^{(n)}, \ldots, x_4^{(n)} \}= T^{(n)}\cap \large( \cup_{i=1}^4 B(o_i, R_n)\large)^c.$$
As  the complementary of the union of open balls of radius greater than $R$ is an $R$-body and  $T^{(n)}$ is convex then   $V^{(n)}$ is an  $R$-body too. The limit of  $W^{(n)}$ is  $W=\{x_1, \ldots , x_4\}$ which is not an $R$-body as proved in Theorem  \ref{lemmasigmao}. \end{proof}

  \begin{theorem}\label{rhulloidedisconesso}
  	Let $ d\ge3;  $ in  $ \E^d  $ there exist connected bodies in a ball of radius  $\sqrt{2}R$ with disconnected R-hulloid. 
\end{theorem}

\begin{proof}  Let us consider the regular simplex $S$  in $\E^d$, defined in  Theorem \ref{lemmasigmao}, with vertices 
	on $ \Sigma_0  := \partial B(0, R_0) $,  $ R_0 := \frac{2 R}{d}   $.

	  The  $(d-2)$ spherical  surface $L_{i,j}:\equiv \pa B(o_i,R)\cap \pa B(o_j,R)$, $i\neq j$, has  center at $\frac{o_i+o_j}{2}$ and contains  $0$. Then,   by \eqref{distcicj}, $L_{i,j}$ has radius 
$$|(o_i+o_j)/2|=	\sqrt{R^2-R^2(\frac{1}{2}+\frac{1}{2d})}=R\sqrt{\frac12-\frac{1}{2d}}.$$
  Then,  
the   maximum distance  of $L_{i,j}$ from $0$ is  
$$2 R\sqrt{\frac12-\frac{1}{2d}}< \sqrt{2}R.$$
	Let  $ D:= (B(0,\sqrt{2}R))^-$  and let 
	\begin{equation}\label{E=example1}
	E := D \cap \left( \bigcup_{j =1} ^{d+1} B(o_j, R) \cup \{0\} \right)^c.
	\end{equation}
	\indent
	Claim 1: {\em $E$ is connected.}\par
	
	First let us consider the  $(d-1)$ spherical balls $U_i=B(o_i)\cap \pa B(0,\sqrt{2}R)$ centered at $c_i= \sqrt{2}o_i$. As $0\in \pa B(o_i,R)$, then by elementary geometric arguments, the spherical radius of $U_i$ is $\frac{\pi}{4}\sqrt{2}R$. By 
\eqref{distcicj},  the spherical  distance between $o_i$ and $o_j$ on $\pa B(0,R)$
is  
$$2R\arcsin \sqrt{\frac12+\frac{1}{2d}} > \frac{\pi}{2}R.$$
Then, the spherical distance between $c_i$ and $c_j$ is greater than $\frac{\pi}{2}\sqrt{2}R$. Since the $(d-1)$ spherical balls $U_i$ have radius $\frac{\pi}{4}\sqrt{2}R$, they   are disjoints and 
$$\mathcal{E}=\pa B(0,\sqrt{2}R) \setminus \cup_{i=1}^{d+1} S_i$$   is a connected subset of $\pa E$.
Let us consider now $x\in E$, then $x \not \in B(o_i,R)$; since $0\in \pa B(o_i,R)$, then $\la x \not \in B(o_i)$ for $\la \geq 1$. Therefore the segment connecting $x$ to $ \sqrt{2}\frac{x}{|x|}R\in  \mathcal{E} $ is a subset of  $E$.  Claim 1 follows.

	Claim 2: {\em $E^c$   is an open set.}
	
	 As $E^c=  D^c \cup \left( \bigcup_{j =1}^{d+1} B(o_j, R) \cup \{0\} \right)$, 	
	it is enough to show that $\{0\}\subset Int (E^c)$. This follows from the fact that $\{0\}$ is in the interior of the simplex $S$, and $Int (S)\subset E^c$.
	
	 Claim 3: {\em The set  of the vertices of $S$ is contained in $E$ }.
	 
	 	For each $i$ the vertex $k_i\in \pa B(o_j,R), j\neq i$ and $k_i \not \in B(o_i, R)^-$.

      $E$  is a closed set from Claim 2;   from Claim 3  and \eqref{E=example1}   it follows that $E$ is not an $R$-body, since     any open ball of radius $R$, containing  $0\in  E^c$, 
  cannot be contained  in $E^c$. 
  
     Claim 4: {\em The point $0 $ has a positive distance from $E$.}
     
     Let us consider for $i=1, \ldots, d+1$ the simplexes $S_i=co(\{0,k_1,\ldots, k_{i-1}, k_{i+1}, k_{d+1}\})$. Then $S=\cup_i S_i$. Let $0< \epsilon< \dist(0,S_i)$, where $S_i$ are the facets of $S$; as $B(0, \epsilon)\subset \cup_i B(0,\epsilon)\cap S_i$, then 
     $$\dist( 0, E) \geq \epsilon.$$

  Let us consider now the body 
  $E \cup \{0\}$.
  Since 
  $$ (E \cup \{0\})^c= D^c \cup \left( \bigcup_{j =1}^{d+1} B(c_j, R)\right), $$
 then $E\cup \{0\}$  is by definition an $R$-body and is the minimal $R$-body containing $E$.
 Then    
     $ co_R(E) = E \cup \{0\}   $
     which  is a not connected set, since is the union of two closed disjoint sets.
 \end{proof}

 \section{R-bodies and  other classes of bodies}
In Remark \ref{reach>=RimpliesRbodies}  it is noticed that the class of $R$-bodies contains the class of  bodies which have reach greater or equal than $R$.

The following class has been introduced   introduced in \cite{Gol}: the class $\mathcal{K}_2^{1/R}$  of bodies  $A$ satisfying the following property:
 \begin{equation}\label{defK21/R}
 \forall x\in A^c \mbox{\, there exists a closed ball \, } D(R)\ni x: D(R)\cap int(A)=\emptyset.
 \end{equation}

 \begin{theorem}\label{inclusionbodies}
The following strong inclusion holds:
\begin{equation}\label{stronginclusion}
 R\mbox{-bodies} \subsetneq \mathcal{K}_2^{1/R}.
\end{equation}
Moreover let   $ A \in \mathcal{K}_2^{1/R}$ and $A=(int(A))^-$, then: \par
i) if $d=2$, then $A$ is an $R$-body; 
\par 
ii) if $d>2 $ , then $A$ can be  not an $R$-body. 
 \end{theorem}
 \begin{proof}
 The  inclusion \eqref {stronginclusion} is obvious: since if  $A$ is an $R$-body and  $x\in A^c$, then  $x\in B(R)$ and $B(R)\cap A=\emptyset$; therefore $\pa B(R) \cap int(A)= \emptyset$. Then, if  $x\in D(R)=\pa B(R)\cup B(R)$ thus $D(R)\cap int(A) =\emptyset.$ 
  The inclusion is strong: let $E= D(0,r)\cap B(0,R)^c \cup \pa B(0,r_1)$, with
 $r_1 < R < r.$ Then $E$ is not an $R$-body as if $x\in B(0,R)\setminus \pa B(0,r_1)$ there is no ball $B \subset E^c$ containing $x$; on the other hand $E \in  \mathcal{K}_2^{1/R}$.
 
 Let $d=2$ and 
  $ A \in \mathcal{K}_2^{1/R}$, $A=(int(A))^-$.
     By contradiction, if $A$ is not an $R$-body, then, by Theorem \ref{proPcor}, there exist $a_1,a_2,a_3 \in A$ such that there exists 
  $z\in co_R(\{a_1,a_2,a_3\})\cap A^c $. Since  $z\neq a_i, i=1,2,3$, then $co_R(\{a_1,a_2,a_3\})$ strongly contains $\{a_1,a_2,a_3\}$;  by Corollary \ref{corollario5.3} and by  Theorem   \ref{cor(v)inplane} 
  $$co_R(\{a_1,a_2,a_3\})= \{a_1,a_2,a_3\}\cup  \tilde{T},$$
  where $\tilde{T}$ is the curvilinear triangle  with vertices $a_i,  i=1,2,3$.  Then $z \in \tilde{T}\cap A^c$; since   $A^c$ is open, then there exists $\tilde{z}\in int( \tilde{T})\cap A^c$. As $\tilde{z}\in int( \tilde{T})$, every  balls $D(R) \ni \tilde{z}$  contains at least one of the vertices $a_i$ in its interior, let $a_1$. Then $D(R)$ contains a neighborhood $U$ of  $a_1$. Since $A=(int(A))^-$, $a_1$ can not be an isolated point of $A$, and in $U$ there are points of $int(A)$.  Therefore property  \eqref{defK21/R} does not hold for $\tilde{z}\in A^c$ and $A \not \in \mathcal{K}_2^{1/R}$, contradiction.
  
  In case ii), let us consider   the set $E$  defined by \eqref{E=example1}  of  Theorem \ref{rhulloidedisconesso}.  $E$ is not an $R$-body but $E\cup \{0\}$ is it.  Then any point of $E^c$, different from $0$ satisfies property \eqref{defK21/R}; moreover
  $$int(E)=int(D)\cap_{j=1}^{d+1}D(o_j,R)^c \cap  \{0\}^c ,$$
 then    $0$ satisfies property \eqref{defK21/R} too,  since the closed ball $D(o_1,R)$ does not  intersect $int(E)$. 
  Then $E\in \mathcal{K}_2^{1/R}$ and $ E$  is not an $R$-body. It easy to see that $E= (int(E))^-$.      
 \end{proof}

\begin{figure}\centering
\def\svgwidth{10cm}
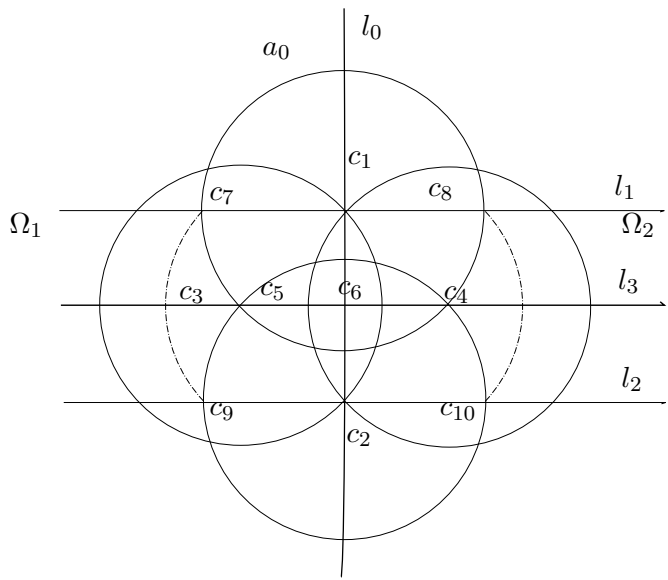
\caption{Case $|c_1-c_2| < \sqrt{3}R$ } 
\label{figura}
\end{figure}

\section*{Appendix}
A detailed  and corrected proof of Proposition \ref{propco2}

{\bf Theorem}\, \cite[Theorem 4.8]{Col} {\em Let $A \subset \E^2 $ be a connected subset of $B(x_0,R)$, $R > 0$. Then A admits R-hull}.

\begin{proof}
We argue by contradiction: if $A$ does not admit $R$-hull, by  Proposition \ref{theorem4.4} reach $(A')_R< R;$ then consequently to the proof of  \cite[Theorem 3.8,]{Col}, it follows that there exist $c_1,c_2 \in A'_R$ such that $ c_1 \ne c_2 $ and
 $$\mathfrak{h}(c_1,c_2) \cap A'_R= \{c_1, c_2\}.$$
   Thus every $x\in \mathfrak{h}(c_1,c_2)\setminus \{c_1,c_2\}$ satisfies $\dist(A,x) < R$. Let $\{c_3,c_4\}= \pa B(c_1)\cap \pa B(c_2)$(see Fig.1 in the case 
  $|c_1-c_2| < \sqrt{3}R$); the case where  $|c_1-c_2| \geq \sqrt{3}R$ has the same proceding.
  
   Let $a \in A$; as $ c_j \in A'_R$ ($j=1,2$), then $ |a-c_j| \geq R$, so 
  $a \in \Om_0:=(B(c_1))^c \cap (B(c_2))^c$.  
 Let $l_0$ be the line joining $c_1,c_2$, let $l_k$ be the line orthogonal to $l_0$, through $c_k$, $k=1,2$. Let $S$ be the strip bounded by $l_1,l_2$ and let $\Om_j$ be the component of the disconnected set $int(S \cap \Om_0)$, nearest  to $c_{j+2}$  $(j=1,2)$.

 Let $c_{4+j}$ be the point of $\pa \mathfrak{h}(c_1,c_2) $ nearest to  $c_{2+j}$ $(j=1,2)$; 
 let $c_{4+2k+j}$ be the point of $l_k\cap \pa B(c_k)$ nearest to $c_{2+j}$ $(k,j=1,2)$ (see Fig.1).
 
 Notation: if $j=1$, then $j^*=2$; if $j=2$, then $j^*=1$. 
 Let us notice that 
 the quadrilaterals with vertices $c_k, c_{4+2k+j}, c_{2+j^*}, c_{4+j} $ are  rhombi with sides of length R, for all  $k,j=1,2$, so $ |c_4-c_5| = |c_3-c_6| = R.$

Then points of $ A, $ with distance less than $ R $ from $ c_{5}$,
must exist and lie in  $ \Omega_1 $; similarly, points of $ A, $ with distance less than $ R $ from $ c_{6}$,
must exist and lie in  $ \Omega_2 $. Thus both $ \Omega_1,  \Omega_2 $ contain a point of $ A.$
As $A$ is connected, then $l_0\cap \Om_0$ contains a point $a_0 \in A$. With no restriction  one may assume that $|a_0-c_2| > |a_0-c_1|$. Then there should be a point of $A$ in each of two half lines of $l_1 \cap \Om_0$. This is impossible: these  two points would have distance greater or equal than $2R$.
\end{proof}

\section*{Acknowledgment}
This work is dedicated to our unforgettable  friend Orazio Arena.

\noindent This work has been partially supported by INDAM-GNAMPA(2022).

\end{document}